\documentclass[12pt]{amsart}
\usepackage{amsmath}
\usepackage{amssymb,latexsym}
\usepackage{mathrsfs}		
\usepackage[curve]{xypic}
\usepackage{color}		
\usepackage{tikz}
\usetikzlibrary{arrows,decorations.pathmorphing,backgrounds,positioning,fit,petri,patterns}
\usepackage{enumerate}

\usepackage{hyperref}
\usepackage{type1cm}
\numberwithin{equation}{section}

\newtheorem{defn}{Definition}[section]
\newtheorem{theorem}{Theorem}[section]

\newtheorem{corollary}[theorem]{Corollary}

\newtheorem{lemma}[theorem]{Lemma}
\newtheorem{prop}[theorem]{Proposition}
\newtheorem{remark}[theorem]{Remark}

\def \begineq{\begin{equation}}
\def \endeq{\end{equation}}

\def \bb{\mathbb}

\def \RR{{\bb{R}}}

\def \ZZ{{\bb{Z}}}

\def \({\left(}
\def \){\right)}
\def \<{\langle}
\def \>{\rangle}
\def \bar{\overline}

\def \Dsum{\bigoplus}
\def \dsum{\oplus}

\def \tensor{\otimes}
\def \union{\cup}

\def \varleq{\leqslant}

\def \xto{\xrightarrow}

\def \img{{\rm img}}
\def \ind{{\rm ind}}

\begin{document}
\title{On the Topology of the Double Spherical Pendulum}

\author{Shengda Hu}
 \address{Department of Mathematics
Wilfrid Laurier University
75 University Avenue West,
Waterloo Ontario, Canada, N2L 3C5 }
\email{shu@wlu.ca }

 \author{Eduardo Leandro}
 \address{Universidade Federal de Pernambuco, Depto de Matematica, Av. Prof. Luiz Freire S/N, Recife, 50740-540, Brazil. }
 \email{eduardo@dmat.ufpe.br}

\author{Manuele Santoprete }
\address{Department of Mathematics
Wilfrid Laurier University
75 University Avenue West,
Waterloo Ontario, Canada, N2L 3C5 }
\email{msantopr@wlu.ca}

\begin{abstract}
This paper studies the topology of the constant energy surfaces of the double spherical pendulum. 
\end{abstract}
\keywords{Double Spherical Pendulum, Topology of Level Sets,  Morse Theory, Gysin Sequence}
\maketitle


\section{Introduction}
The double spherical pendulum was studied in \cite{Marsden}, where several dynamical features of the problem were brought to light using techniques of geometric mechanics and bifurcation theory with symmetry.


In this paper we consider  a different aspect of the problem:  we study the topology of the constant energy surfaces of the double spherical pendulum. 

The  topology of energy surfaces is important for understanding  Hamiltonian systems, and is strictly related to some dynamical features.

For instance, McCord and Meyer \cite{Mccord2000} found that the existence of global cross sections, a standard tool to extract global information about the flow, depends on the topology of the energy levels. An  exaustive study  of the existence of global cross sections for compact energy surfaces of two-degrees of freedom Hamiltonians can be found in \cite{Bolsinov}.
 
Moreover  some Hamiltonian systems, but not all,  are geodesic flows. McCord, Meyer and Offin \cite{Mccord2002} showed that there are some topological conditions that need to be satisfied for an Hamiltonian system to be a geodesic flow.

Another important problem is the study of the {\it integral manifolds } $I_w$, i.e. the common level sets of the integrals of the dynamical system \cite{SmaleI,SmaleII}. The integral manifolds have the property that if $v\in I_w$ then the orbit through $v$ is contained in $I_w$. 
In general such orbit can be very complicated, however, as suggested by Smale \cite{SmaleI,SmaleII}, a crude,  but very important, invariant of this orbit is the topological type of $I_w$. 
In the case of integrable Hamiltonian systems, the Arnol'd-Liouville theorem states that all non-singular compact, connected 
integral manifolds are tori. Consequently such tori constitute a (singular) foliation of the constant energy surfaces. 
The situation is not so simple  in the non-integrable case, where very little is known about the orbits, especially when there are several degrees of freedom.  
In such cases, understanding the topology of the energy levels seems to be a first step toward describing the integral manifolds of the system and how they foliate the constant energy surfaces.

One of the goals of this paper is to give a better understanding of the topology of the energy surfaces of the double spherical pendulum. 


A second goal of the paper is to illustrate different techniques that can be used in studying the topology of energy surfaces. In our case we are able to give a complete description of the non-critical energy surfaces through Morse theory and algebraic topology. In particular, we show how to compute topological invariants such as the homology groups and the Betti numbers when an explicit description is too difficult to obtain.

A third goal is to explain the table with the Betti numbers given in \cite{Mccord2002} and to show that such table  is valid for a large open set of parameters describing the double spherical pendulum.

The paper is organized as follows. In the next section we describe the mechanical system and derive its Hamiltonian. In Section 3  we  give an explicit description of the non-critical energy surfaces. In the following section we compute the Betti numbers using the results of the preceding section. 
In  Section 5 we compute the Betti numbers without using the explicit description of the energy surfaces. In the last section we describe some dynamical consequences of the topology. 
\section{Set-up}
Consider the mechanical system consisting of two coupled spherical pendula in a gravitational field.
Let $r_1$ be the position vector  of  the first pendulum relative to its pivot $O$, and $r_2$ the position of the second pendulum with respect to the first. Let $(m_i, l_i)$, $i=1,2$, be the mass and the length of the $i$-th pendulum, respectively. Then we have

\begin{equation}\label{eq:multipendula}
\left\{\begin{matrix}
r_0 & = & 0,\\
h_0 & = & 0, \\
r_i & = & l_i q_i\\
q_i & = & \begin{pmatrix} x_i \\ y_i \\ z_i
\end{pmatrix}, \\
h_i & = & l_i  z_i + h_{i-1},\\
\end{matrix}\right. 
\end{equation}
where $q_i$ is the unit vector in the direction of $r_i$, and $h_i$ is the height of the $i$-th pendulum with respect to $O$. 
Thus the configuration space is $Q := S^2 \times S^2$. The Lagrangian $L(q, \dot q, t)$ is given by
$$L(q_1, q_2, \dot q_1, \dot q_2, t) := \frac{1}{2}(m_1 ||l_1\dot q_1||^2 + m_2 ||l_1\dot q_1 + l_2\dot q_2||^2) - V(q_1, q_2)$$
where 
$$V(q_1, q_2) = g(m_1h_1 + m_2h_2) = gm_2 l_2\left(k z_1 + z_2 \right)$$
and $k := \frac{(m_1+m_2)l_1}{m_2l_2}$ is called the \emph{slope} of the double spherical pendulum. The choice of the term \emph{slope} comes from the following presentation of the image of the set $\{V < c\}$ as the shaded area in Figure \ref{fig:slope}.

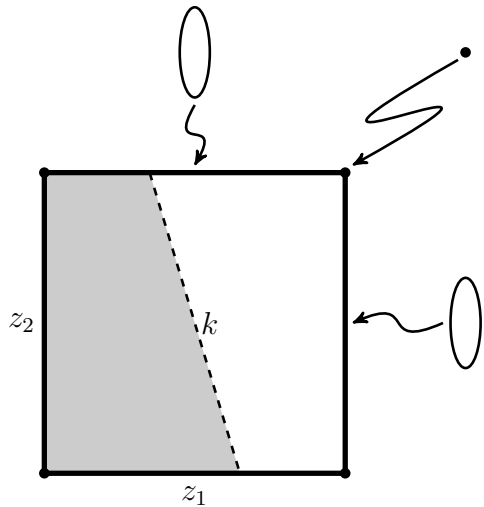
\begin{figure}[!h]
\begin{center}
\begin{tikzpicture}[scale=2,>=stealth']
\fill[color=gray!40!white] (-1, -1) -- (-1, 1) -- (-0.3, 1) -- (0.3, -1) -- cycle;
\draw[line width=2pt] (-1, -1) rectangle (1,1);
\fill (-1, -1) circle (1pt)
		(-1, 1) circle (1pt)
		(1, -1) circle (1pt)
		(1, 1) circle (1pt);
\draw[dashed,line width=1pt] (-0.3, 1) -- (0.3, -1);
\draw[line width=1pt] (0, 1.8) ellipse [x radius=0.1, y radius=0.3];
\draw[line width=1pt] (1.8, 0) ellipse [x radius=0.1, y radius=0.3];
\fill (1.8, 1.8) circle (1pt);
\draw [line width=1pt,->] (0,1.45) .. controls (-0.2, 1.1) and (0.2, 1.4) .. (0,1.05);
\draw [line width=1pt,->] (1.65, 0) .. controls (1.2, -0.2) and (1.5, 0.2) .. (1.05, 0);
\draw [line width=1pt,->] (1.75,1.75) .. controls (0, 0.75) and (2.75, 2) .. (1.05,1.05);
\draw (0.1,0) node {$k$}
      (-1.15,0) node {$z_2$}
      (0, -1.15) node {$z_1$};
\end{tikzpicture}
\caption{The image of the two height functions $z_1$ and $z_2$, with $k > 1$, and some of its pre-images in $S^2 \times S^2$.} \label{fig:slope}
\end{center}
\end{figure}


The conjugate momenta are given by
$$p_1 = \frac{\partial L}{\partial \dot q_1} = m_1 l_1^2\dot q_1 + m_2l_1(l_1\dot q_1 + l_2\dot q_2) \text{ and } p_2 = \frac{\partial L}{\partial \dot q_2} = m_2l_2 (l_1\dot q_1+ l_2\dot q_2).$$
From this, it follows that
$$\dot q_1 = \frac{1}{l_1}\left(\frac{p_1}{m_1l_1} - \frac{p_2}{m_1l_2}\right) \text{ and } \dot q_2 = \frac{1}{l_2}\left(\frac{p_2}{m_2l_2} - \frac{p_1}{m_1l_1} + \frac{p_2}{m_1l_2}\right)$$
and 
$$H(p_1, p_2, q_1, q_2) = \frac{1}{2}\left(\frac{1}{m_1}\left|\left|\frac{p_1}{l_1} - \frac{p_2}{l_2}\right|\right|^2 + \frac{1}{m_2}\left|\left|\frac{p_2}{l_2}\right|\right|^2\right) + V(q_1, q_2).$$
Notice that the Hamiltonian flow must satisfy the following four constraints:
\[g_1: ~q_1\cdot q_1=1, ~ g_2:~ q_2\cdot q_2=1,~ g_3:~ q_1\cdot \dot q_1=q_1\cdot\left(\frac{p_1}{m_1l_1^2} - \frac{p_2}{m_1l_1l_2}\right) =0,\]
and
\[
g_4:~q_2\cdot \dot q_2=q_2\cdot  \left(\frac{p_2}{m_2l_2^2} - \frac{p_1}{m_1l_1l_2} + \frac{p_2}{m_1l_2^2}\right)=0.
\]
The phase space of the system is $M := T^*Q = T^*(S^2 \times S^2)$ and $H$ defines a function on $T^*Q$.
\section{Topology of energy surfaces}
In this section, we will first apply Morse theory to get some information on the energy surfaces. Next, we describe these surfaces explicitly.
\subsection{Description using Morse theory}
\begin{prop}\label{prop:MorseFunc}
$V(q_1, q_2)$ is a Morse function on the configuration space $Q$.
\end{prop}
\begin{proof}
We note that $z_i$ is the height function on the $i$-th $S^2$ which is evidently a Morse function. It follows that $kz_1 + z_2$ is necessarily a Morse function on $S^2 \times S^2$ whenever $k \neq 0$. Thus $V$ is a Morse function on $Q$.
\end{proof}

We can write down the $4$ critical points of $V$ as follows. Let $0$ and $\infty$ denote the lowest and highest point of $S^2$ (in terms of the height function), then the critical points of $V$ is given by
$$P_1 = (0, 0), \ \ P_2 = (0, \infty), P_3 = (\infty, 0)\ \  \text{ and } \ \ P_4 = (\infty, \infty).$$
The Morse index of $V$ at $P_i$ is given by the following:
$$\ind_Q P_1 = 0, \ \ \ind_Q P_2 = \ind_Q P_3 = 2 \ \  \text{ and } \ \ \ind_Q P_4 = 4.$$
The slope determines the relative position of $P_2$ and $P_3$:
$$k > 1 \iff V(P_2) < V(P_3) \ \ \text{ and }\ \  k = 1 \iff V(P_2) = V(P_3).$$

\begin{corollary}\label{cor:MorseFunc}
$H$ is a Morse function on the phase space $M = T^*Q$.
\end{corollary}
\begin{proof}
Let $P = (p, q) \in T^*Q$ be a critical point of $H$. From the explicit expression, we see that $q \in Q$ must be a critical point of $V$, which is  a Morse function on $Q$. Thus $P$ is non-degenerate along $Q$. Since $H - V$ is a non-degenerate quadratic form defined on the fibers of $T^*Q \to Q$ (independent of $q \in Q$), we see that $P$ is non-degenerate in $T^*Q$. It follows that $H$ is a Morse function on $M$.
\end{proof}

From the above analysis, we also see that there are precisely $4$ critical points of $H$, of the form $(0, q)$, where $q$ is a critical point of $V$ on $Q$. Thus we will again denote them $P_i$ with $i = 1, 2, 3, 4$ as above. Since $H-V$ is positive definite on each fiber, we have
$$\ind_M P_i = \ind_Q P_i.$$
According to Morse theory, the topology of level sets of $H$ change when passing a critical value by attaching handles, which are determined by the Morse index at the corresponding critical points. Let $M(a) := H^{-1}(a)$. We note that for small $\varepsilon$ we have $M(\varepsilon) \simeq S^7$, as $P_1$ is an isolated minimum. For $\rho >> 1$, we have $M(\rho) \simeq T_1^*Q \simeq T_1 Q$ the unit (co)tangent bundle of $Q$, since $H^{-1}((-\infty, \rho))$ is now a tubular neighbourhood of the $0$-section. The interesting critical points $P_2$ and $P_3$  both have index $2$. We have

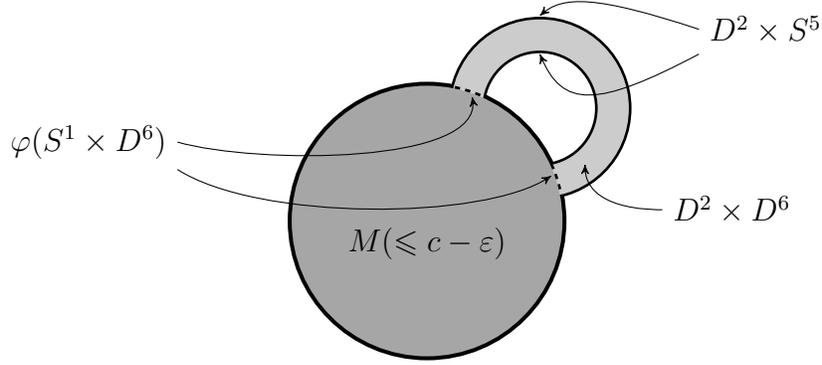
\begin{figure}[!t]
\begin{center}
\begin{tikzpicture}[scale=1.5,>=stealth']
\draw[line width=3pt] (-0.5, -0.5) circle (1.2);
\fill[even odd rule,color=gray!40!white] (0.5, 0.5) circle (0.5)
                                         (0.5, 0.5) circle (0.8);
\draw[dotted,line width=2pt] (0.7, -0.5) arc (0:80:1.2);
\draw[line width=1pt] (0.5, 0.5) circle (0.5)
                      (0.5, 0.5) circle (0.8);
\fill[color=gray!70!white] (-0.5, -0.5) circle (1.2);
\draw (-0.5, -0.7) node {$M(\varleq c - \varepsilon)$}
      (2.5, 1.2) node[name=handle_border] {$D^2 \times S^5$}
      (2.2, -0.4) node[name=handle_body] {$D^2 \times D^6$}
      (-3.5, 0.2) node[name=handle_hidden] {$\varphi(S^1 \times D^6)$};
\draw [->] (handle_body.west) .. controls +(left:0.1) and +(down:0.4) .. (0.9,0);
\draw [->] (handle_border.west) .. controls (1, 1.5) and (0.7, 1.5) .. (0.5, 1.32);
\draw [->] (handle_border.south west) .. controls (1, 0.5) and (0.7, 0.5) .. (0.5, 0.98);
\draw [->] (handle_hidden.east) .. controls (-2, 0) and (0, 0) .. (-0.1, 0.6);
\draw [->] (handle_hidden.south east) .. controls (-2, -0.5) and (0, -0.5) .. (0.6, -0.1);
\end{tikzpicture}
\caption{The solid lines represent $(M(c - \varepsilon) \setminus \varphi (S^1 \times D^6)) \union_\partial (D^2 \times S^5)$. The dashed lines represents $ \varphi (S^1 \times D^6) $. \label{fig:morse} }
\end{center}
\end{figure}

\begin{prop}
Let $c = H(P_2)$ or $H(P_3)$ and $\varepsilon > 0$ small, then
\begin{equation}\label{eq:morse}
M(c+ \varepsilon) \cong (M(c - \varepsilon) \setminus (S^1 \times D^6)) \union_\partial (D^2 \times S^5),
\end{equation}
where $\union_\partial$ denotes gluing along the boundary $S^1 \times S^5$. 
\end{prop}
\begin{proof}
For $a \in \RR$, let $M(\varleq a) = H^{-1}((-\infty, a))$. Then Morse theory describes the topological change of $M(\varleq a)$ when passing the critical level $c$ as the attaching of a handle:
$$M(\varleq c+ \varepsilon) \cong M(\varleq c - \varepsilon) \union_\varphi (D^2 \times D^6),$$
where $\partial(D^2 \times D^6) = (S^1 \times D^6) \union (D^2 \times S^5)$ and $\varphi : S^1 \times D^6 \to \partial M(\varleq c - \varepsilon)$ is an embedding. This also provides a description of topological change of the level sets $M(a)$, since
$$M(a) = \partial M(\varleq a).$$
Thus
\[\begin{split}M(c+\varepsilon) = \partial M(\varleq c+ \varepsilon) \cong \partial \left(M(\varleq c - \varepsilon) \union_\varphi (D^2 \times D^6)\right)\\ 
\cong (M(c - \varepsilon) \setminus \varphi (S^1 \times D^6)) \union_\partial (D^2 \times S^5)
\end{split}
\] (see Figure \ref{fig:morse}).
Since $\varphi$ is an embedding, the result follows.
\end{proof}

\noindent
We write for a small $\varepsilon > 0$:
$$M_i := M(H(P_i) + \varepsilon) \text{ for } i = 1, 2, 3, 4.$$
Then $M_1 \cong S^7$ and $M_4 \cong T_1^*Q$. 

We now give a more explicit description of $M_2$ and $M_3$ using equation (\ref{eq:morse}).

Note that 
$$S^7 = \partial D^8 \simeq \partial (D^2 \times D^6) = (S^1 \times D^6) \union_\partial (D^2 \times S^5).$$
In this case, since $S^7$ is simply connected, the surgery giving rise to $M_2$ depends only on the gluing along the boundary, and is not affected by where it is performed, thus
$$M_2 \cong ((S^1 \times D^6) \union_\partial (D^2 \times S^5)\setminus (S^1 \times D^6)) \union_\partial (D^2 \times S^5)\cong(D^2 \times S^5) \union_\partial (D^2 \times S^5).$$

\begin{figure}[!t]
\begin{center}
\begin{tikzpicture}[scale=1.5,>=stealth']
\draw[line width=1pt,rounded corners=12pt] (-1.45, 0) --  (-1.15, 0.75) -- (-0.5, 1) -- (0.15, 0.75) -- (0.8,0.8) -- (1.25, 1) -- (1.7,0.8) -- (2, 0.25) -- (1.7, -0.3) -- (1.25, -0.5) -- (0.8, -0.3) -- (0.15, -0.75) -- (-0.5, -1) -- (-1.15, -0.75)-- cycle;
\draw (-0.5, 0.95) arc (90:270:0.25 and 0.95); 
\draw[dashed] (-0.5, -0.95) arc (-90:90:0.25 and 0.95);
\fill[even odd rule,color=gray!40!white] (1.25, 0.15) circle (0.15)
                                         (1.25, 0.15) circle (0.3);
\draw[dashed,line width=1pt] (1.25, 0.15) circle (0.15)
                             (1.25, 0.15) circle (0.3);
\draw (2.25, 0) node {$\cong$};
\draw[line width=1pt,rounded corners=12pt,xshift=4cm] (-1.45, 0) --  (-1.15, 0.75) -- (-0.5, 1) -- (0.15, 0.75) -- (0.7, 0.3) -- (2.25, 0.3) -- (2.8,0.8) -- (3.25, 1) -- (3.7,0.8) -- (4, 0.25) -- (3.7, -0.3) -- (3.25, -0.5) -- (2.8, -0.3) -- (2.25, 0.1) -- (0.7, 0.1) -- (0.15, -0.75) -- (-0.5, -1) -- (-1.15, -0.75)-- cycle;
\draw[xshift=4cm]  (-0.5, 0.95) arc (90:270:0.25 and 0.95);
\draw[dashed,xshift=4cm] (-0.5, -0.95) arc (-90:90:0.25 and 0.95);
\draw (5.5, 0.3) arc (90:270:0.05 and 0.1);
\draw[dashed] (5.5, 0.1) arc (-90:90:0.05 and 0.1);
\draw[dashed] (7.25, -0.44) arc (-90:90:0.2 and 0.68);
\fill[even odd rule,color=gray!40!white,xshift=4.2cm] (3.25, 0.15) circle (0.15)
                                                    (3.25, 0.15) circle (0.3);
\draw[dashed,line width=1pt,xshift=4.2cm] (3.25, 0.15) circle (0.15)
                                        (3.25, 0.15) circle (0.3);
\draw (7.25, 0.92) arc (90:270:0.2 and 0.68);
\draw (0.5, 1.2) node {$M_2$}
      (3.5, 1.2) node {$M_2$}
      (5.5, 1.2) node {$\#$}
      (7.25, 1.2) node {$S^7$}
      (3.5, -2) node[name=loci] {$S^1 \times D^6$}
      (6, -1) node[name=neck] {$S^6$};
\draw[->] (loci.west) .. controls (2, -2) and (1.5, -1) .. (1.4, 0);
\draw[->] (loci.east) .. controls (6, -2) and (7, -1) .. (7.35, -0.05);
\draw[->] (neck.north) -- (5.55, 0.2);
\end{tikzpicture}
\caption{\label{fig:surgery} }
\end{center}
\end{figure}
%
%

Similarly, since one can show that $M_2$ is simply connected, the surgery giving rise to $M_3$ again depends only on the gluing along the boundary. Moreover, we may describe this surgery as a connected sum. Note that $M_2 \simeq M_2 \# S^7$ and $S^1 \times D^6$ can isotopied 
into $S^7$ (see Figure \ref{fig:surgery}) and the surgery on $S^7$ as above gives $M_3 \simeq M_2 \# N_2$ with
$$N_2 \cong (D^2 \times S^5) \union'_\partial (D^2 \times S^5)$$
where $\union'_\partial$ means a possibly different gluing from the one used for $M_2$. One needs to determine the attaching maps between the $D^2\times S^5$'s in order to obtain $M_2$ and $M_3$.

\subsection{Explicit construction}
 We first show $M_1 \cong S^7$ in a more explicit fashion. Without any loss of generality, we assume from now on $k > 1$.  The accessible region in this case is $U_1 = \pi(M_1) \subset Q$, where $\pi : T^*Q \to Q$ is the bundle projection. Using a diagram as in Figure \ref{fig:slope}, we may represent $U_1 \cong D^4$ as in Figure \ref{fig:firstlevel}:\\
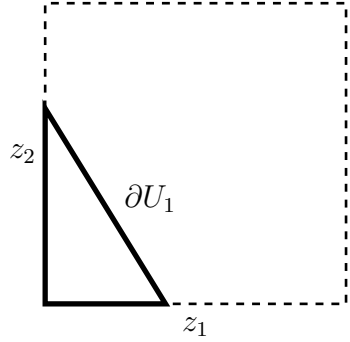
\begin{figure}[!h]
\begin{center}
\begin{tikzpicture}[scale=2]
\draw[line width=2pt] (-1, -1) -- (-1, 0.3) -- (-0.2,-1) -- cycle;
\draw[dashed,line width=1pt] (-1, 0.3) -- (-1, 1) -- (1, 1) -- (1, -1) -- (-0.2, -1);
\draw (-0.3,-0.3) node {$\partial U_1$}
      (-1.15,0) node {$z_2$}
      (0, -1.15) node {$z_1$};
\end{tikzpicture}
\caption{$U_1 \subset S^2 \times S^2$, with $k > 1$\label{fig:firstlevel}}
\end{center}
\end{figure}

\noindent
From the fact that the kinetic energy $H-V$ is positive definite, it follows that $M_1 \cong T_1U_1 / \sim_\partial$, where $T_1$ denotes the unit (co)tangent bundle and $\sim_\partial$ denote the equivalence relation that identifies all the points along each fiber over points in $\partial U_1$. Here, the fibers are all homeomorphic to $S^3$. Since $U_1$ is contractible, we see that $T_1U_1 \cong D^4 \times S^3$. We define explicitly a map $D^4 \times S^3 \to S^7$ which identifies all the points along each fiber over points in $\partial D^4$. Let $D^4$ be the unit ball in $\RR^4$ and $S^3$ the unit sphere in $\RR^4$. We take $S^7$ also as a unit sphere in $\RR^8$, then the following is the desired map
\[\begin{split}D^4 \times S^3 \to S^7 : ((x_4, x_5, x_6, x_7), (x_0, x_1, x_2, x_3)) \mapsto \\
\left(\sqrt{1-\sum_{i = 4}^7x_i^2}\ (x_0, x_1, x_2, x_3), (x_4, x_5, x_6, x_7)\right).\end{split}\]

\begin{prop}\label{prop:levelset2}
$M_2 \cong S^2 \times S^5$.
\end{prop}
\begin{proof}
Let $U_2 := \pi(M_2) \subset Q$ be the accessible region. Then 
$$U_2 = \{(q_1, q_2) | gm_2l_2(kz_1 + z_2) + V_0 \varleq V(P_2) + \varepsilon\}.$$
Recall that $M_2 \cong T_1U_2 / \sim_\partial$.
%
Let $\pi_i : Q \to S^2$ be the projection to the $i$-th factor. Since $k > 1$, we see that $\pi_2|_{U_2} : U_2 \to S^2$ is surjective, and the fiber of $\pi_2|_{U_2}$ at each $q_2 \in S^2$ is homeomorphic to a disc. Thus we have $U_2 \cong D^2 \times S^2$. It follows that $M_2 \cong T_1(D^2 \times S^2)/\sim_\partial$.

Notice that $S^2 \times D^2$ can be embedded into $\RR^4$ as follows. From $S^2 \times \RR \cong \RR^3\setminus\{(0,0,0)\}$ we see that $S^2 \times \RR^2 \cong \RR^4\setminus\{(0,0,0, w) | w \in \RR\}$. Then $S^2 \times D^2 \subset S^2 \times \RR^2$ is embedded into $\RR^4$. It follows that $T(D^2 \times S^2)$ is trivial, i.e.
 $$T(D^2 \times S^2) \cong D^2 \times S^2 \times \RR^4.$$

Thus, we have
$$M_2 \cong  T_1(D^2 \times S^2)/\sim_\partial\cong (D^2 \times S^2 \times S^3) /\sim_\partial.$$
We see $M_2 \cong S^2 \times S^5$ as follows. Let $S^5$ be the unit sphere in $\RR^6$,  given by $\sum_{i = 0}^5 x_i^2 = 1$, $S^3$ be unit sphere in $\RR^4$ given by $\sum_{i = 0}^3 x_i^2 = 1$ and $D^2$ as the unit disk in $\RR^2$ given by $x_4^2 + x_5^2 \varleq 1$. Then we have the map
$$D^2 \times S^3 \to S^5 : ((x_4, x_5), (x_0, x_1, x_2, x_3)) \mapsto (\sqrt{1-x_4^2 - x_5^2}\ (x_0, x_1, x_2, x_3), (x_4, x_5)).$$
The induced map $D^2 \times S^2 \times S^3 \to S^2 \times S^5$ is precisely the quotient $\sim_\partial$.
\end{proof}


Let's relate the explicit construction in the proof to the description using Morse theory. Let $\partial'(D^2 \times D^2) = S^1 \times D^2 \subset \partial(D^2 \times D^2)$ and notice that
\begin{equation}\label{eq:s2decomp}
U_2 \cong D^2 \times S^2 \cong (D^2 \times D^2) \union_{\partial'} (D^2 \times D^2),
\end{equation}
where $\union_{\partial'}$ means that the two copies of $D^2 \times D^2$ are glued along $\partial'(D^2 \times D^2)$. A similar map as used in the proof above shows that
$$T_1(D^2 \times D^2) / \sim_{\partial'} \cong D^2 \times D^2 \times S^3/ \sim_{\partial'} \cong D^2\times S^5,$$
where $\sim_{\partial'}$ is the equivalence relation that identifies all the points along each fiber over points in $S^1 \times D^2 = \partial' (D^2 \times D^2)$. Thus it follows from \eqref{eq:s2decomp} that
$$M_2 \cong (D^2 \times S^5) \union_\partial (D^2 \times S^5).$$
Let $\varphi$ denote the attaching map $S^1 \times S^5 \to S^1 \times S^5$. We present this in Figure 
\ref{fig:secondlevel}:
\begin{figure}[!h]
\begin{center}
\begin{tikzpicture}[scale=2]
\draw[line width=2pt] (-2.5, -1) -- (-2.5,1) -- (-1.8, 1) -- (-1.2, -1) -- cycle;
\draw[dashed,line width=1pt] (-1.8, 1) -- (-0.5,1) -- (-0.5,-1) -- (-1.2, -1);
\fill[color=gray!40!white] (0.5, 0) -- (0.5, 1) -- (1.5, 1) -- (1.5, 0) -- cycle;
\draw[line width=2pt] (0.5, -1) -- (0.5,1) -- (1.5, 1) -- (1.5, -1) -- cycle;
\draw[dashed,line width=1pt] (1.5, 1) -- (2.5,1) -- (2.5,-1) -- (1.5, -1);
\draw[dashed,line width=1pt] (0.45, 0) -- (1.55, 0);
\draw (-1.3,0) node {$\partial U_2$}
      (-2.65,0) node {$z_2$}
      (-1.5, -1.15) node {$z_1$}
      (0,0) node {$\cong$}
      (1, -1.15) node {$D^2 \times S^2$}
      (1.65, 0) node {$\partial$};
\end{tikzpicture}
\caption{$U_2 \subset S^2 \times S^2$, with $k > 1$\label{fig:secondlevel}}
\end{center}
\end{figure}
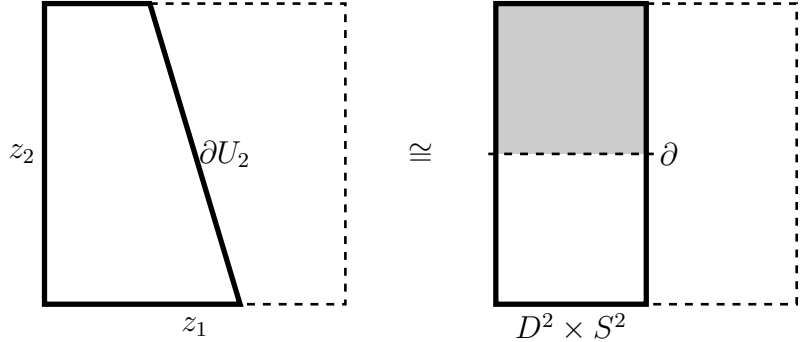
\noindent \\
\noindent
In the right figure, the edge labeled by $\partial$ represents the boundary $S^1 \times S^2$. The points in the fiber of $\pi$ over points in the boundary are identified. The dashed line corresponds to $D^2 \times S^1$. The shaded and unshaded part each corresponds to $D^2 \times D^2 \subset D^2 \times S^2$.

\begin{prop}\label{prop:levelset3}
$M_3 \cong (S^2 \times S^5)^{\#2}$.
\end{prop}
\begin{proof}
Again, let $U_3 := \pi(M_3)$ be the accessible region in this case. Then 
$$U_3 = \{(q_1, q_2) | gm_2l_2(kz_1 + z_2) + V_0 \varleq V(P_3) + \varepsilon\}.$$
We note that, since $V$ is a Morse function, $U_3 \cong Q \setminus \mathring{D^4}$, since $P_4$ is the maximal critical point of $V$. The left diagram in Figure \ref{fig:thirdlevel} represents $U_3$.
\begin{figure}[!h]
\begin{center}
\begin{tikzpicture}[scale=2]
\draw[line width=2pt] (-2.5, -1) -- (-2.5,1) -- (-1.1, 1) -- (-0.5, 0.2) -- (-0.5, -1) -- cycle;
\draw[dashed,line width=1pt] (-1.1, 1) -- (-0.5,1) -- (-0.5, 0.2);
\fill[color=gray!40!white] (0.5, 0) -- (0.5, 1) -- (0.9, 1) -- (0.9, 0) -- cycle;
\fill[color=gray!40!white] (1.5, -1) -- (1.5, -0.6) -- (2.5, -0.6) -- (2.5, -1) -- cycle;
\draw[line width=2pt] (0.5, -1) -- (0.5,1) -- (0.9, 1) -- (0.9, -0.6) -- (2.5, -0.6) -- (2.5,-1) -- cycle;
\draw[dashed,line width=1pt] (0.9, 1) -- (2.5,1) -- (2.5,-0.6);
\draw[dashed,line width=1pt] (0.4, 0) -- (1, 0);
\draw[dashed,line width=1pt] (1.5, -1.1) -- (1.5, -0.5);
\draw (-1,0.5) node {$\partial U_3$}
      (-0.7, 0.8) node {$D^4$}
      (-2.65,0) node {$z_2$}
      (-1.5, -1.15) node {$z_1$}
      (0,0) node {$\cong$}
      (1.1, -0.4) node {$\partial$}
      (1.7, 0.4) node {$D^2 \times D^2$};
\end{tikzpicture}
\caption{$U_3 \subset S^2 \times S^2$, with $k > 1$\label{fig:thirdlevel}}
\end{center}
\end{figure}
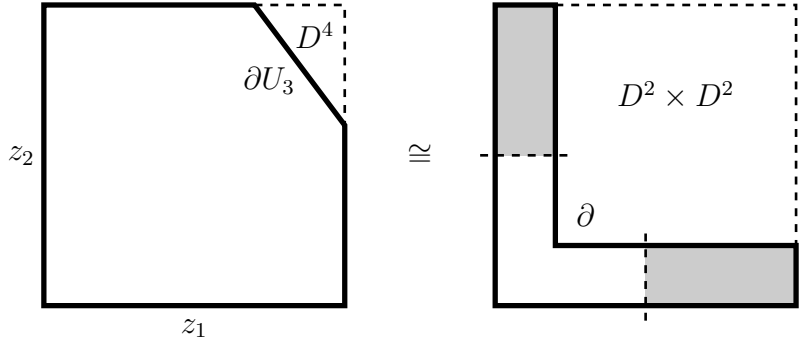
We show next that $TU_3 = TQ|_{U_3}$ is trivial.  Recall that, in Proposition \ref{prop:levelset2}, we proved that $T(D^2\times S^2)=T(Q)|_{D^2\times S^2}$ is trivial. Consequently   $TQ|_{\{x\} \times S^2}$ and $TQ|_{S^2 \times \{x\}}$ are trivial for any $x \in S^2$, since they are restrictions of the bundle $T(Q)|_{D^2\times S^2}$. 
On the other hand  $U_3$ has the same homotopy type of $S^2 \vee S^2$. This can be shown in two ways. First  consider the thick L shaped figure  in the right  diagram of Figure \ref{fig:thirdlevel}. This represents $U_3$
since $D^4 \cong D^2\times D^2$.  
The thick letter L can be contracted to a thin letter L, where the horizontal segment of the letter  coincides with the bottom side of the square, and the vertical segment coincide with the left side of the square. The thin letter L is precisely $S^2 \vee S^2$. Second consider the negative gradient flow (e.g. with respect to the standard round metric on $S ^2$) of $V$. Such flow  gives a contraction of $U_3$ to $S^2 \vee S^2$. Moreover contraction preserves the triviality of restriction of vector bundles (this follows from  \cite{Husemoller} Theorem 4.7, page 30).

 On the other hand, the dashed lines divide the right figure into $3$ parts, each of which is homeomorphic to $D^4 \cong D^2\times D^2$.
The union of any of the two  shaded areas with the unshaded one is  homeomorphic to $D^2 \times S^2$.
Comparing Figure \ref{fig:thirdlevel} to Figure \ref{fig:secondlevel}, we see that the attaching maps at both dashed lines (in Figure \ref{fig:thirdlevel}) are identical to $\varphi$, the map that defines $M_2$. From the discussion using Morse theory, we see that $N_2 \cong S^2 \times S^5$ and $M_3 \cong M_2 \# N_2 \cong (S^2 \times S^5)^{\# 2}$.
\end{proof}
\section{The Betti numbers}
In this section we compute the Betti numbers of the energy surfaces, from the explicit description given in the previous section.
\subsection{Properties of homology groups}\label{subsec:homologygroups}
Let $X$ be a manifold  of dimension $n$, then we  denote by $H_k(X; G)$ the $k$-th homology group  of $X$ and by $H_*(X;G) = \dsum_{i \in \ZZ} H_i(X;G)$ the homology of $X$, where $G$ is an abelian group. In the following we will just write $H_k(X)$ to denote the $k$-th homology group  of $X$ with coefficients in $\RR$.  We now recall some facts on the homology groups (with coefficients in $\RR$). These can be found in most standard textbooks, for example \cite{Hatcher, Munkres, Spanier}.

\begin{theorem}[Properties of homology groups]\label{thm:prophomology}
Let  $X$ be a manifold  of dimension $n$, then  the homology groups with  coefficients in $\RR$ are finite dimensional vector spaces which have the following properties:
\begin{enumerate}
\item $H_i(X) = 0$ if $i \notin [0, n]$.
\item If $X$ is connected, then $H_0(X) \cong \RR$.
\item ({\bf Poincar\'e duality}) If $X$ is orientable and closed, then 
$$H_i(X) \cong H_{n-i}(X).$$
\item If $X$ is contractible (to a point), then $H_*(X) = H_0(X) \cong \RR$. In particular, $H_*(pt) \cong \RR$.
\item ({\bf K\"unneth formula}) If $Y$ is another manifold, of dimension $m$, then 
$$H_k(X \times Y) \cong \Dsum_{i + j = k} H_i(X) \tensor H_j(Y).$$
\item If $X = S^k$, then $H_i(S^k) = 0$ if $i \neq 0, k$, and $H_0(S^k) \cong \RR \cong H_k(S^k)$.
\item ({\bf Connected sum}) If $X = X_1 \# X_2$ is a connected sum of two manifolds of the same dimension $n$, then we have
$$H_i(X) \cong H_i(X_1) \dsum H_i(X_2) \text{ for } 0 < i < n.$$
\end{enumerate}
\end{theorem}
The last of the above is a corollary of the Mayer-Vietoris sequence in homology.

\begin{remark}
Note that in this paper we consider homology groups with real coefficients. Many books consider coefficients in $\ZZ$, or in a general Abelian group. 
In the latter two cases the properies of the homology are slightly different, in particular properties (3), (5), and (7) take a  more complicated form.  
\end{remark}

We are interested in the Betti numbers $\beta_i(X)$,i.e, in the dimensions of the $H_i(X)$,
$$\beta_i(X) := \dim_\RR H_i(X).$$
Sometimes we drop the space $X$ from the notation $\beta_i(X)$ if the space is clear. The K\"unneth formula then yields the
 following property of Betti numbers:
$$\beta_k(X \times Y) = \sum_{i+j = k} \beta_i(X) \beta_j(Y).$$
For example, since $\beta_0(S^2) = \beta_2(S^2) = 1$, and $\beta_1(S^2)=0$ by K\"unneth formula, we have
$$\beta_0(S^2 \times S^2) = \beta_4(S^2 \times S^2) = 1, ~ \beta_2(S^2 \times S^2)=2,  \text{ and } \beta_1(S^2 \times S^2)=\beta_3(S^2 \times S^2)=0 .$$
We can write down the Betti numbers of the energy surface $M_1 \cong S^7$ by using Theorem \ref{thm:prophomology} part (6), and by using the K\"unneth formula, we can write those of $M_2 \cong S^2 \times S^5$. We  obtain 
\begin{itemize}
\item for $M_1$, $\beta_0 = \beta_7 = 1$, and all others vanish,
\item for $M_2$, $\beta_0 = \beta_2 = \beta_5 = \beta_7 = 1$, and all others vanish.
\end{itemize}
Since $M_3 \cong (S^2 \times S^5)^{\# 2}$, using the connected sum property, we get
\begin{itemize}
\item for $M_3$, $\beta_0 = \beta_7 = 1, \beta_2 = \beta_5 = 2$, and all others vanish.
\end{itemize}

\subsection{The Gysin sequence} 
\begin{defn}
 We say that $p: X \to B$ is a sphere bundle with fiber $S^k$ if the following holds
\begin{enumerate}[(a)]
\item at each point $b \in B$, $p^{-1}(b) \cong S^k$
\item and there is a neighbourhood $U$ of $b$ such that $p^{-1} (U) \cong U \times S^k$ such that the map $p$ corresponds to the projection to the first factor under this homeomorphism.
\end{enumerate}
\end{defn}

The simplest sphere bundle is the product $B \times S^k$, for which the K\"unneth formula provides the Betti numbers. More generally, we have the Gysin sequence.
\begin{theorem}[Gysin sequence]
If $p : X \to B$ is a sphere bundle with fiber $S^k$ and $k > 0$, then
$$\ldots \to H_i(X) \xto{p_*} H_i(B) \xto{\Psi_\xi} H_{i-k-1}(B) \xto{} H_{i-1}(X) \xto{p_*} H_{i-1}(B) \to \ldots$$
is a long exact sequence of homology groups.
\end{theorem}
\noindent
The map $\Psi_\xi$ is defined from the Thom isomorphism and more details can be found in \cite{Spanier} (pp. 260). In particular, when $X$ is the unit tangent bundle, $\Psi_\xi$ coincides with multiplication by the Euler number of $B$. 
We recall that a sequence of maps between linear spaces
$$\ldots \xto{f_{i+1}} V_{i+1} \xto{f_i} V_i \xto{f_{i-1}} V_{i-1} \xto{} \ldots $$ 
is called \emph{exact} if
$$\ker f_i \cong \img f_{i+1} \text{ for all } i \in \ZZ.$$
If all but $3$ consecutive terms in an exact sequence are zero, e.g.
$$0 \to U \to V \to W \to 0,$$
we say that it is a \emph{short exact sequence}. For a short exact sequence, we have
$$\dim_\RR V = \dim_\RR U + \dim_\RR W.$$
If all but $2$ consecutive terms in an exact sequence are zero, e.g.
$$0 \to U \xto{f} V \to 0,$$
we have $\ker f = 0$ and $\img f = V$, i.e. $f$ is an isomorphism.

We'll compute the Betti numbers of $M_4 \cong T_1Q$ and rederive the results obtained in \cite{Mccord2002} (pp. 1241). The unit (co)tangent bundle $T_1 Q$ is a sphere bundle with fiber $S^3$. Thus the Gysin sequence gives
$$\ldots \to H_i(M_4) \xto{\pi_*} H_i(Q) \xto{\Psi_\xi} H_{i-4}(Q) \xto{} H_{i-1}(M_4) \xto{\pi_*} H_{i-1}(Q) \to \ldots$$
The computation of the Betti numbers for  $S^2\times S^2$ given in the previous section  implies that  $H_i(Q) = 0$, unless $i = 0, 2, 4$. Consequently, the Gysin sequence splits into simpler sequences: 
\begin{itemize}
\item $0 \to H_i(M_4) \xto{\pi_*} H_i(Q) \to 0$ is exact for $i =0,1,2$, and
\item $0 \to H_4(M_4) \xto{\pi_*} H_4(Q) \xto{\Psi_\xi} H_0(Q) \xto{} H_3(M_4) \to 0$ is exact.
\end{itemize}
Together with Poincar\'e duality, we then have 
$$\beta_i(M_4) = \beta_i(Q) = \beta_i(S^2 \times S^2) \text{ for } i \neq 3, 4.$$
The map $\Psi_\xi : \RR \cong H_4(Q) \to H_0(Q) \cong \RR$ is given by multiplication of the Euler number of $Q$:
$$e(Q) = \sum_{i = 0}^4 (-1)^i \beta_i(Q) = 4.$$
It follows that 
$$\ker \Psi_\xi = 0 \text{ and } \img \Psi_\xi = H_0(Q).$$
By the exactness of the sequence, we see that
$$H_4(M_4) \cong \img \pi_* \cong \ker \Psi_\xi = 0.$$
Either by exactness of the sequence again or via Poincar\'e duality, we get $H_3(M_4) \cong H_4(M_4) = 0$.
We have the following table of Betti numbers:
\begin{table}[!h]\caption{The table of Betti numbers for the double pendulum\label{table:Bettinumbers}.}
\begin{center}
\begin{tabular}{|c|c|c|c|c|c|c|c|c|}
\hline
Space & $\beta_0$ & $\beta_1$ & $\beta_2$ & $\beta_3$ & $\beta_4$ & $\beta_5$ & $\beta_6$ & $\beta_7$ \\
\hline
$M_1$ & 1 & 0 & 0 & 0 & 0 & 0 & 0 & 1 \\
\hline
$M_2$ & 1 & 0 & 1 & 0 & 0 & 1 & 0 & 1 \\
\hline
$M_3$ & 1 & 0 & 2 & 0 & 0 & 2 & 0 & 1 \\
\hline
$M_4$ & 1 & 0 & 2 & 0 & 0 & 2 & 0 & 1 \\
\hline
\end{tabular}
\end{center}
\end{table}
\subsection{Integer coefficients}
With exactly  the  same arguments we may find the homology groups with coefficients in $\ZZ$.
In this case  the homology groups are finitely generated abelian groups, so they have the form $H_k(X;\ZZ)\cong \ZZ\oplus\ZZ\oplus\ldots\oplus\ZZ\oplus T$, where $T$ is an abelian group containing finite many elements and is called the torsion part of $H_k(X;\ZZ)$. The $i$-th Betti number of $X$ is the number of copies of $\ZZ$ in $H_i(X)$, and the torsion part is not captured by them. The only torsion in the homology groups we computed lies in
$$H_3(M_4) \cong \ZZ_4.$$
Thus we have the list of homology group with integer coefficients:
\begin{table}[!h]
\caption{The table of homology groups with integer coefficients for the double pendulum\label{table:homology}.}
\begin{center}
\begin{tabular}{|c|c|c|c|c|c|c|c|c|}
\hline
Space & $H_0$ & $H_1$ & $H_2$ & $H_3$ & $H_4$ & $H_5$ & $H_6$ & $H_7$ \\
\hline
$M_1$ & $\ZZ$ & 0 & 0 & 0 & 0 & 0 & 0 & $\ZZ$ \\
\hline
$M_2$ & $\ZZ$ & 0 & $\ZZ$ & 0 & 0 & $\ZZ$ & 0 & $\ZZ$ \\
\hline
$M_3$ & $\ZZ$ & 0 & $\ZZ\dsum \ZZ$ & 0 & 0 & $\ZZ\dsum \ZZ$ & 0 & $\ZZ$ \\
\hline
$M_4$ & $\ZZ$ & 0 & $\ZZ\dsum \ZZ$ & $\ZZ_4$ & 0 & $\ZZ\dsum \ZZ$ & 0 & $\ZZ$ \\
\hline
\end{tabular}
\end{center}
\end{table}
\begin{remark}\label{rmk:coefficients}
The change of coefficients from $\RR$ to $\ZZ$ makes the statements of some topological facts more complicated. Here, for example, the K\"unneth formula will have extra terms coming from the torsion parts of the homology groups. In the statement of Poincar\'e duality, the cohomology groups will no longer be isomorphic to the homology groups. Nonetheless, the exactness of the Gysin sequence (which holds for any coefficients) provides $H_3(M_4) \cong \ZZ_4$ in the above table.
\end{remark}

\begin{remark}
One may ask if the Betti numbers for homology with different coefficients coincide. As a matter of fact they do in the case of $G=\ZZ$ and $\RR$. This is a consequence of the Universal Coefficient Theorem \cite{Hatcher}.
\end{remark}

\section{Alternative computation}
In this section we describe a computation of the Betti numbers without the explicit knowledge of the energy surfaces. This approach was suggested  to the authors by Chris McCord \cite{McCord}. Since some of the fibers of the bundles we are considering are pinched to points we cannot use the Gysin sequence directly. However there is a  generalized Gysin sequence for relative homology groups. This sequenece will be our main tool to obtain the Betti numbers. 
\subsection{Properties of relative homology}
Let $X$ be a manifold of dimension $n$ and $A \subset X$ a submanifold (open or closed), then the  $k$-th relative homology group (with coefficients in $\RR$) of the pair $(X, A)$ is denoted $H_k(X, A)$. The readers are referred again to standard textbooks, e.g. \cite{Hatcher, Munkres, Spanier}, for the properties given below.

\begin{theorem}[Properties of relative homology] \label{thm:proprelhomology}
Let $X$ be a manifold of dimension $n$ and $A \subset X$ a submanifold (open or closed), then the relative homology groups (with coefficients in $\RR$) of the pair $(X, A)$ have the following properties:
\begin{enumerate}
\item ({\bf Long exact sequence for the pair}) There is a long exact sequence of homology groups:
\[\label{eq:longex}
\ldots \to H_k(A) \to H_k(X) \to H_k(X, A) \xto{\partial_*} H_{k-1}(A) \to \ldots
\]
\item ({\bf Excision}) 
Given subspaces  $V \subset A\subset X$ such that the closure of $V$ is contained in the interior of $A$, i.e. $\bar V \subset int(A)$, then
$$H_k(X, A) \cong H_k(X \setminus V, A \setminus V)$$
for all $k$.
\item ({\bf Contraction}) Suppose that $B \subset X$ is a submanifold which contracts to $A$, then
$$H_k(X, A) \cong H_k(X, B).$$
for all $k$.
\item ({\bf K\"unneth formula}) If $Y$ is another manifold, of dimension $m$, then
$$H_k(X \times Y, A \times Y) \cong \dsum_{i+j = k} H_i(X, A) \tensor H_j(Y).$$
\item ({\bf Relative to a point}) Suppose that $H_0(X) \cong \RR$, i.e. $X$ has only one component, then for a point $pt \in X$, we have:
$$H_k(X, pt) \cong H_k(X) \text{ for } k \neq 0 \text{ and } H_0(X, pt) \cong 0.$$
\item ({\bf Boundary}) If $X$ is a manifold of dimension $n$, with boundary $\partial X \neq \emptyset$, then the map $\partial_* : H_n(X, \partial X) \to H_{n-1}(\partial X)$ in part (1) of the theorem  is an isomorphism.
\end{enumerate}
\end{theorem}

The facts we stated above are in the form that we will be using, instead of the most general form, cf. Remark \ref{rmk:coefficients}. The readers are referred to the standard textbooks for a complete treatment and detailed proofs.

\subsection{Relative  homology of the pairs $(U_2,\partial U_2)$ and $(U_3,\partial U_3)$}
To find the homology of the pairs $(U_2,\partial U_2)=(S^2\times D^2, S^2\times S^1)$ and $(U_3,\partial U_3)=(S^2\times S^2\setminus D^4, S^3)$
 the following lemmas are required.

\begin{lemma}\label{lemma:relativeapp} Let $D^4 \subset Q$ be a $4$-dimensional ball and $\mathring D^4$ its interior. Then we have
$$H_k(Q \setminus \mathring D^4, \partial(Q \setminus \mathring D^4) ) \cong H_k(Q, pt) \cong \left\{\begin{matrix}
\RR^2, & \text{ for } k = 2 \\ \RR, & \text{ for } k = 4 \\ 0, & \text{ for the other } k.
\end{matrix}\right.$$ 
\end{lemma}
\begin{proof}
First, we note that $D^4$ can be contracted to its center in $Q$, thus
$$H_k(Q, pt) \cong H_k(Q, D^4),$$ for all $k$,
by the contraction property above. 
The excision property implies that
$$H_k(Q, D^4) \cong H_k(Q \setminus \mathring D^4, D^4 \setminus \mathring D^4) \cong H_k(Q \setminus \mathring D^4, \partial(Q \setminus \mathring D^4) ),$$
for all $k$.
The result then follows from Theorem \ref{thm:proprelhomology} part (5). 
\end{proof}

\begin{lemma}
$H_k({D^2}, \partial {D^2}) \cong \RR$ for $k = 2$ and vanishes for all other $k$.
\end{lemma}
\begin{proof}
 Note that $\partial {D^2} \cong S^1$.
We use the long exact sequence for the pair (Theorem \ref{thm:proprelhomology}, part (1)):
$$0 \to H_2({D^2}, S^1) \xto{\partial_*} H_1(S^1) \to H_1({D^2}) \cong 0 \to $$
$$ \to H_1({D^2}, S^1) \to H_0(S^1) \xto{\cong} H_0({D^2}) \to H_0({D^2}, S^1) \to 0.$$
It follows that $H_2({D^2}, S^1) \cong H_1(S^1) \cong \RR$ and $H_1({D^2}, S^1) \cong H_0 ({D^2}, S^1) \cong 0$.
\end{proof}

Now we can determine the homology  groups for the pairs $(U_2, \partial U_2)$ and $(U_3, \partial U_3)$:

\begin{prop}\label{prop:relativehomology}
Recall that  $U_2 \cong S^2 \times D^2$,  $\partial U_2 \cong S^2 \times S^1$, $U_3 = Q \setminus D^4$,  and  $\partial U_3 = S^3$. Then
\begin{enumerate}[(a)]
 \item \[H_k(U_2, \partial U_2) \cong \left\{\begin{matrix}
\RR, & \text{ when } k = 2, 4 \\ 0, & \text{ for all other } k.
\end{matrix}\right.\]
\item \[ H_k(U_3, \partial U_3) \cong \left\{\begin{matrix}
\RR^2, & \text{ when } k = 2 \\
\RR, & \text{ when } k = 4 \\ 0, & \text{ for all other } k.
\end{matrix}\right.\]
\end{enumerate}
\end{prop}
\begin{proof}
\begin{enumerate}[(a)]
\item  By K\"unneth formula, we have
$$ H_k(U_2, \partial U_2) \cong \Dsum_{i + j = k} H_i(S^2) \tensor H_j(D^2, \partial D^2)\cong \left\{\begin{matrix}
\RR & \text{ when } k = 2, 4 \\ 0 & \text{ for all other } k.
\end{matrix}\right.$$
\item  The proof follows from  Lemma \ref{lemma:relativeapp}.
\end{enumerate}
\end{proof}

\subsection{Relative Gysin sequence}
We now introduce a Gysin sequence for relative homology group. This is a generalized version of the standard Gysin sequence. 

\begin{theorem}[Relative Gysin sequence]\label{thm:relativeGysin}
Let $p : X \to B$ be an orientable  fiber bundle with  fiber the $k$-sphere ($S^k$), with $k\geq 1$, and path-connected base space
$B$. If $C\subset B$ and $K=p^{-1}(C)$ then the homology sequence  for the pair  
$$\ldots \to H_i(X, K) \xto{p_*} H_i(B, C) \to H_{i-k-1}(B, C) \xto{} H_{i-1}(X, K) \xto{p_*} H_{i-1}(B, C) \to \ldots$$
is exact. 
\end{theorem}
\noindent
See \cite{Spanier} (pp. 483) for more details, including a proof of the theorem. 
We will use the relative Gysin sequence to compute the homology groups of the energy surfaces $M_i$, for $i = 2, 3$.
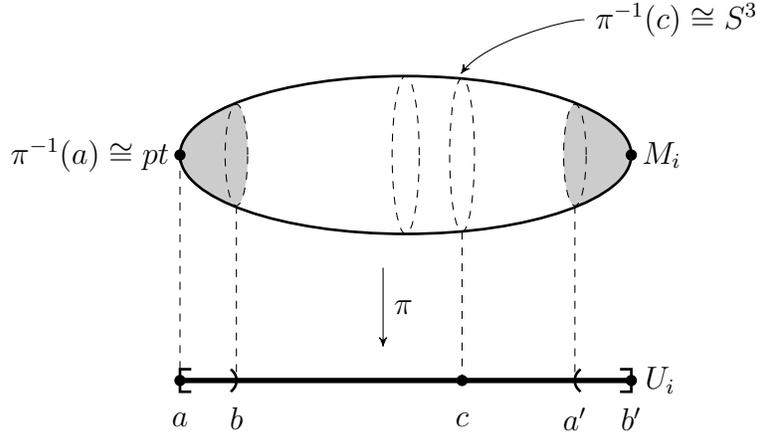
\begin{figure}[!t]
\begin{center}
\begin{tikzpicture}[scale=1.5,>=stealth']
\begin{scope}
\clip (0,0) ellipse [x radius=2,y radius=0.7];
\fill[color=gray!40!white] (-1.5, -0.45) arc (-90:90:0.1 and 0.45) -- (-1.5, 0.5) -- (-2, 0.5) -- (-2,-0.5) -- (-1.5, -0.5) -- cycle;
\fill[color=gray!40!white] (1.5, 0.45) arc (90:270:0.1 and 0.45) -- (1.5, -0.5) -- (2, -0.5) -- (2,0.5) -- (1.5, 0.5) -- cycle;
\end{scope}
\draw[line width=1pt] (0,0) ellipse [x radius=2,y radius=0.7];
\draw[dashed] (0,0) ellipse [x radius=0.12,y radius=0.7];
\draw[dashed] (0.5, 0) ellipse [x radius=0.11,y radius=0.68];
\draw[dashed] (1.5, 0) ellipse [x radius=0.1, y radius=0.45];
\draw[dashed] (-1.5, 0) ellipse [x radius=0.1, y radius=0.45];
\draw[line width=2pt] (-2, -2) -- (2, -2);
\fill (-2,0) circle (0.05) node[anchor=east] {$\pi^{-1}(a) \cong pt$}
      (2, 0) circle (0.05) node[anchor=west] {$M_i$}
      (-2, -2) circle (0.05)
      (2, -2) circle (0.05) node[anchor=north,inner sep=10pt] {$b'$} node[anchor=west,inner sep=5pt] {$U_i$}
      (0.5, -2) circle (0.05);
\draw[dashed] (-2,0) -- (-2, -2) node[anchor=north,inner sep=12pt] {$a$};
\draw[dashed] (0.5,-0.68) -- (0.5, -2) node[anchor=north,inner sep=12pt] {$c$};
\draw[dashed] (1.5,-0.45) -- (1.5, -2) node[anchor=north,inner sep=10pt] {$a'$};
\draw[dashed] (-1.5,-0.45) -- (-1.5, -2) node[anchor=north,inner sep=10pt] {$b$};
\draw[line width=1pt] (-1.9, -1.9) -- (-2, -1.9) -- (-2, -2.1) -- (-1.9, -2.1);
\draw[line width=1pt] (1.9, -1.9) -- (2, -1.9) -- (2, -2.1) -- (1.9, -2.1);
\draw[line width=1pt,rounded corners=5pt] (-1.55,-2.1) -- (-1.45, -2) -- (-1.55, -1.9);
\draw[line width=1pt,rounded corners=5pt] (1.55,-2.1) -- (1.45, -2) -- (1.55, -1.9);
\draw[->] (-0.2, -1) -- (-0.2, -1.7);
\draw (-0.2,-1.35) node[anchor=west] {$\pi$};
\draw (2.4,1.2) node[name=level] {$\pi^{-1}(c) \cong S^3$};
\draw[->] (level.west) .. controls (1.4,1.2) and (0.7, 1) .. (0.49, 0.72);
\end{tikzpicture}
\caption{A diagrammatic representation of the bundle $\pi:M_i\to U_i$. The union of the intervals $[a,b)$ and $(a',b']$ represents $V_i$, while the endpoints of $[a,b']$ represent $\partial U_i$ (note that in our case the sets  $\partial U_i$ and $V_i$ are connected, even if in the figure they are not). The  shaded areas of the ellipsoid represent $W_i=\pi^{-1}(V_i)$, and the poles represent $\pi^{-1}(\partial U_i)$. The preimage of an interior point of $U_i$ under $\pi$ is a three-sphere $S^3$ (represented by a circle), and that of a point on the boundary  $\partial U_i$ is a point.\label{fig:excision} }
\end{center}
\end{figure}

Recall that for $i = 2, 3$, $M_i \cong T_1 U_i / \sim_\partial$, where $T_1 U_i$ is a $3$-sphere bundle over $U_i$. Let $\pi : M_i \to U_i$ denote the map induced by the projection of the (co)tangent bundle $TU_i \to U_i$, then
$$\pi|_{\partial U_i} : \pi^{-1}(\partial U_i) \xto{\cong} \partial U_i.$$

Note that the energy surfaces $M_i$  for $i = 2, 3$ are not quite sphere bundles (since some of the fibers are pinched to points) and thus we cannot apply directly the preceding sequence. The following lemma shows  that even if  the relative Gysin sequence does not apply  to the pairs $(U_i,\partial U_i)$, $(M_i,\pi^{-1}(\partial U_i))$ we can apply it to slightly modified spaces that have the same relative homology of the  foregoing  pairs. 

\begin{lemma}\label{lemma:excision}
Assume $M_i\setminus\pi^{-1}(\partial U_i)$ is a fiber bundle with fiber $S^3$ and base $U_i\setminus\partial U_i$ then 
 $$H_k(U_i, \partial U_i) \cong H_k(U_i \setminus \partial U_i, {V_i} \setminus \partial U_i) \text{ and }$$
$$H_k(M_i, \pi^{-1}(\partial U_i)) \cong H_k(M_i \setminus \pi^{-1}(\partial U_i), W_i \setminus \pi^{-1}(\partial U_i)),$$
for all $k$, where  $V_i \subset U_i$ is a collar neighbourhood of $\partial U_i$\footnote{The boundary of a compact manifold with boundary always has a collar neighbourhood.}, and $W_i=\pi^{-1}(V_i)$.
\end{lemma}
\begin{proof}

Since $V_i$ is a collar neighbourhood of $\partial U_i$, then  $V_i \cong \partial U_i \times [0, \varepsilon)$  and it contracts to $\partial U_i$. It follows that $W_i=\pi^{-1}(V_i)$ contracts to $\pi^{-1}(\partial U_i)$. 
Thus $H_i(U_i,\partial U_i)\cong H_i(U_i,V_i)$  and $H_i(M_i,W_i)\cong H_i(M_i,\pi^{-1}(\partial U_i))$ .
Now we can apply the excision property to the triples $M_i, \pi^{-1}(V_i) , \pi^{-1}(\partial U_i)$ and $U_i,V_i,\partial U_i $. Consequently, $H_k(U_i, \partial U_i) \cong H_k(U_i,  V_i) \cong H_k(U_i \setminus \partial U_i, {V_i} \setminus \partial U_i)$, and $H_k(M_i, \pi^{-1}(\partial U_i)) \cong H_k(M_i, W_i) \cong H_k(M_i \setminus \pi^{-1}(\partial U_i), W_i \setminus \pi^{-1}(\partial U_i))$.
\end{proof}

To simplify notations, we write
\[(X_i, K_i) :=(M_i \setminus \pi^{-1}(\partial U_i), W_i \setminus \pi^{-1}(\partial U_i ))
\text{ and }
(B_i, C_i) := (U_i \setminus \partial U_i, {V_i} \setminus \partial U_i)\]

Since, in our case,  $X_i$ is a fiber bundle with fiber $S^3$ and base $B_i$,  $C_i\subset B_i$ and $K_i=p^{-1}(C_i)$, we can apply the relative Gysin sequence.  
 
\subsection{Homology of $M_2$}

Substituting the relative homology groups $H_k(U_2,\partial U_2)\cong H_k(B_2, C_2)$,  computed in Proposition \ref{prop:relativehomology} part (a),   into the relative Gysin sequence for the $3$-sphere bundle $(X_2, K_2) \to (B_2, C_2)$ we obtain
\begin{itemize}
\item for $k = 0, 1, 2, 3, 4$, 
$$0 \cong H_{k-3}(B_2, C_2) \to H_k(X_2, K_2) \xto{\pi_*} H_k(B_2, C_2) \to H_{k-4}(B_2, C_2) \cong 0$$ 
$$\Longrightarrow \ H_k(X_2, K_2) \cong H_k(B_2, C_2).$$
\item for $k = 5, 6, 7$, 
$$0 \cong H_{k+1}(B_2, C_2) \to H_{k-3}(B_2, C_2) \to H_k(X_2, K_2) \xto{\pi_*} H_k(B_2, C_2) \cong 0$$
$$\Longrightarrow \ H_k(X_2, K_2) \cong H_{k-3}(B_2, C_2).$$
\end{itemize}
Summarizing, by the reasoning above and Lemma \ref{lemma:excision}, we have
$$H_k(M_2, \pi^{-1}(\partial U_2)) \cong H_k(X_2, K_2) \cong \left\{\begin{matrix}
\RR, & \text{ when } k = 2, 4, 5, 7, \\ 0, & \text{ for all other } k.
\end{matrix}\right.$$

Since $\pi^{-1}(\partial U_2) \cong S^2 \times S^1$, using the sequence in part  (1) of Theorem \ref{thm:proprelhomology} for the pair $(M_2, \pi^{-1}(\partial U_2))$, we have
\begin{itemize}
\item for $k = 5, 6, 7$, 
$$0 \cong H_k(\pi^{-1}(\partial U_2)) \to H_k(M_2) \to H_k(M_2, \pi^{-1}(\partial U_2)) \to H_{k-1}(\pi^{-1}(\partial U_2)) \cong 0$$
$$\Longrightarrow \ H_k(M_2) \cong H_k(M_2, \pi^{-1}(\partial U_2)).$$
\item for $k = 3, 4$,
$$0 \cong H_4(\pi^{-1}(\partial U_2)) \to H_4(M_2) \to H_4(M_2, \pi^{-1}(\partial U_2)) \xto{\partial_*} $$
$$\xto{\partial_*} H_3(\pi^{-1}(\partial U_2)) \to H_3(M_2) \to H_3(M_2, \pi^{-1}(\partial U_2)) \cong 0.$$
We show in Proposition \ref{prop:reliso} that $\partial_*$ is an isomorphism. It follows that 
$$H_3(M_2) \cong H_4(M_2) \cong 0.$$
\item for $k = 0, 1, 2$, we may apply Poincar\'e duality.
\end{itemize}
Thus, we get
$$H_k(M_2) \cong \left\{\begin{matrix}
\RR, & \text{ when } k = 0, 2, 5, 7, \\ 0, & \text{ for all other } k.
\end{matrix}\right.$$
\subsection{Homology of $M_3$}
Note that $U_3 = Q \setminus D^4$. According to Proposition \ref{prop:relativehomology} part (b), we have
$$H_k(B_3, C_3) \cong H_k(U_3, \partial U_3) \cong \left\{\begin{matrix}
\RR^2, & \text{ when } k = 2, \\
\RR, & \text{ when } k = 4, \\ 0, & \text{ for all other } k.
\end{matrix}\right.$$
and the relative Gysin sequence together with  Lemma \ref{lemma:excision} give
$$H_k(M_3, \pi^{-1}(\partial U_3)) \cong H_k(X_3, K_3) \cong \left\{\begin{matrix}
\RR^2, & \text{ when } k = 2, 5, \\ \RR, & \text{ when } k = 4, 7, \\ 0, & \text{ for all other } k.
\end{matrix}\right.$$
After applying the long exact sequence for the pair $(M_3, \pi^{-1}(\partial U_3))$, modulo the fact that $\partial_* : H_4(M_3, \pi^{-1}(\partial U_3)) \to H_3(\pi^{-1}(\partial U_3))$ is an isomorphism, we get
$$H_k(M_3) \cong \left\{\begin{matrix}
\RR^2, & \text{ when } k = 2, 5, \\ \RR, & \text{ when } k = 0, 7, \\ 0, & \text{ for all other } k.
\end{matrix}\right.$$

Thus, we recovered the computation of Betti numbers obtained in the previous sections without using the explicit description of the manifolds. 
\section{Some dynamical consequences of the topology of the double pendulum}

\subsection{Obstruction to the existence of geodesic flows}
The connection between Hamiltonian flows and geodesic flows was studied in \cite{Mccord2002}, where the authors proved the following
\begin{theorem}
If $P$ is a compact connected orientable manifold of dimension $2n-1,n>2$, and the torsion subgroup $T_{n-1}(P)$ is trivial, then a necessary condition for $P$ to be the unit tangent bundle of some orientable $n$-manifold is
\[
\left|\sum_{i=0}^{n-2}(-1)^i\beta_i(P)+(-1)^n(1-\beta_{2n-2}(P))\right|=1+\beta_1(P)-\beta_n(P).
\]
\end{theorem}

We describe the arguments in \cite{Mccord2002}, where this theorem is applied to the double spherical pendulum. In this case  $P$ is either $M_1$, $M_2$, $M_3$ or $M_4$. 
If $P=M_4$ then the torsion subgroup is $T_3(M_4)=\ZZ_4$ (see Table \ref{table:homology}). In this case, $M_4$ is the unit tangent bundle of $S^2\times S^2$ and, in fact, the flow is the geodesic flow of the Jacobi metric. 

In all the other cases the torsion subgroup is $T_3(M_i)=0$. So the theorem above would require 
the identity
\[
|\beta_0-\beta_1+\beta_2+1-\beta_6|=1+\beta_1-\beta_4
\]
to be satisfied. Since $\beta_0=1$ and $\beta_1=\beta_4=\beta_6 = 0$ for $M_1$, $M_2$, and $M_3$ (see Table \ref{table:Bettinumbers}) the identity 
reduces to $|2+\beta_2|=1$. Since $\beta_2$ is non-negative the identity is never satisfied. Thus for $i=1,2,3$ the flow on $M_i$ is not a geodesic flow.

\subsection{Obstruction to the  existence of global cross sections}

Global cross sections are a standard tool to extract global information about a flow from the return map.
Traditionally the cross section is either a compact manifold without boundary or a compact manifold whose boundary is invariant under the flow. 
We begin with defining the first kind of global cross section:

\begin{defn}
 Let $M$ be a connected manifold of dimension $m$ without boundary, $\Phi:\RR\times M\to M$
a flow, and $C$ a submanifold of $M$ of dimension $(m-1)$ without boundary then $C$ is a {\it global cross section}
if 
\begin{enumerate}
 \item For each point $p\in M$ there is a $t(p)>0$ such that $\Phi(t(p),p)\in C.$
\item There is a continuous function $\tau: C\to \RR$ such that
\begin{enumerate}
\item  $\Phi(t,p)\notin C$ for all $p\in C$ and $0<t<\tau(p)$.
\item $\Phi(\tau(p),p)\in C$ for all $p\in C$.
\end{enumerate}
\item There is an open neighbourhood $U$ of $\{0\} \times C$ in $\RR \times C$ such that $\Phi|_U$ is a homeomorphism from $U$
to an open neighbourhood of $C$ in $M$.
\end{enumerate}
The function $\tau$ is called the {\it return time}. The function $P:C\to C:p\to \Phi(\tau(p),p)$ is called the {\it Poincar\'e map}.
\end{defn}

In \cite{Mccord2000} the authors found some necessary conditions for the existence of global cross sections. In particular they proved:

\begin{theorem}\label{thm:crosssection}
 If the flow $\Phi:\RR\times M\to M$ on the manifold $M$ admits a global
cross section $C$, then
\begin{enumerate}
\item $M$ is a fiber bundle over $S^1$ with fiber $C$.
\item If $C$ is of finite type (i.e. its homology is finitely generated), then $\chi(M)=0$ (the Euler characteristic of $M$ is zero).
\item If $C$ is of finite type, $H_1(M;\ZZ)$ has a factor $\ZZ$.
\item The flow has no equilibrium points. 
\end{enumerate}
\end{theorem}

In the case of the double spherical pendulum it follows from  Table \ref{table:homology} that $H_1(M_i;\ZZ)=0$ for $i=1,2,3,4$. Consequently, by
Theorem \ref{thm:crosssection},  the flow on $M_i$ for $1=1,\ldots, 4$ does not have a global cross section of finite type.

\subsection{Obstruction to integrability}
In this subsection we mention how the topology of the configuration space is related to the integrability of the Hamiltonian system.
We use the following theorem essentially due to Taimanov (\cite{Kozlov}; see also \cite{Taimanov1988a})
\begin{theorem}
Suppose that the configuration space of a natural Hamiltonian system (i.e. the Hamiltonian is kinetic plus potential energy) with $n$ degrees of freedom is a connected analytic manifold $Q$, and the Hamiltonian function is analytic on the phase space. If this system has $n$ independent analytic integrals then 
\[
\beta_k(Q)\leq   \binom {n} {k}, \quad k=1,\ldots n,
\]
where $\beta_k(Q)= \binom {n} {k}$ if $Q$ is the n-dimensional torus. 
If $\beta_1(Q)=n$, then the inequalities are replaced by equalities. 
In particular, for an integrable system we have $\beta_1(Q)\leq n$.
\end{theorem}

In our case $Q=S^2\times S^2$, and  $n$=4. The Betti numbers of $Q$ were obtained in Subsection \ref{subsec:homologygroups}, so we obtain the following inequalities
\[\begin{split} & \beta_0=1\leq  \binom {4} {0}=1, ~\beta_1=0\leq\binom {4} {1}= 4,~\beta_2=2\leq \binom {4} {2}=6,\\
 & \beta_3=0\leq \binom {4} {3}=4,~\beta_4=1\leq\binom {4} {4}= 1\end{split}.\]
Hence, in this case, there is no topological obstruction to the integrability. This is not surprising since one can consider a Hamiltonian of the form $H=H_1+H_2$, where $H_1$ corresponds to  an integrable Hamiltonian on the first $S^2$ (e.g. the free particle constrained to move on a sphere) and  $H_2$ corresponds to  an integrable Hamiltonian on the second $S^2$. 

\section*{Acknowledgments}
The authors acknowledge with gratitude useful discussions pertinent to the present research with 
Chris McCord, Rick Moeckel, and  B. Doug Park. 
The third author was supported by a NSERC Discovery Grant. 
Part of this work was carried out when Eduardo Leandro and Manuele Santoprete were visiting the American Institute of Mathematics, and we are grateful for their hospitality and support.

\appendix
\section{}
\def\thesection{\Alph{section}} 
In this appendix, we show
\begin{prop}\label{prop:reliso}
The map $\partial_* : H_4(M_i, \pi^{-1}(\partial U_i)) \to H_3(\pi^{-1}(\partial U_i))$ is an isomorphism.
\end{prop} 

Before the proof, we need to recall the \emph{naturality} of homology groups. Let $f : (M, A) \to (Y, B)$ be a continuous map, meaning
$$f : X \to Y \text{ is continuous and } f(A) \subset B.$$
Then $f$ induces natural maps between the homology groups:
$$f_* : H_k(X) \to H_k(Y), f_* : H_k(A) \to H_k(B) \text{ and } f_* : H_k(X, A) \to H_k(Y, B),$$
for all $k$.
In this sense, the map $\partial_*$ in the sequence \eqref{eq:longex} in the Theorem \ref{thm:proprelhomology} is also \emph{natural}, namely, $\partial_* \circ f_* = f_* \circ \partial_*$ in the following diagram,
$$
\text{
\xymatrix{
H_k(X, A) \ar[rr]^{\partial_*}\ar[d]_{f_*} && H_{k-1}(A) \ar[d]^{f_*}\\
H_k(Y, B) \ar[rr]^{\partial_*} && H_{k-1}(B)
}
}
$$

\begin{proof}
Consider $\pi : (M_i, \pi^{-1}(\partial U_i)) \to (U_i, \partial U_i)$. Then the naturality above implies $\partial_* \circ \pi_* = \pi_* \circ \partial_*$ in the following diagram:
$$
\text{
\xymatrix{
H_4(M_i, \pi^{-1}(\partial U_i)) \ar[rr]^{\partial_*}\ar[d]_{\pi_*} && H_3(\pi^{-1}(\partial U_i)) \ar[d]^{\pi_*}\\
H_4(U_i, \partial U_i) \ar[rr]^{\partial_*} && H_3(\partial U_i)
}
}
$$
We note that $\pi : \pi^{-1}(\partial U_i) \to \partial U_i$ is a homeomorphism, it follows that the $\pi_*$ on the right is an isomorphism. On the other hand, we note that the bottom map $\partial_*$ is an isomorphism, by the boundary property of relative homology. The computation using relative Gysin sequences also showed the left $\pi_*$ is also isomorphism. It implies that the top $\partial_*$ is an isomorphism as well.
\end{proof}

\end{document}